\documentclass{article}
\usepackage[utf8]{inputenc}
\usepackage{amsthm}
\usepackage{amstext}
\usepackage{amsmath}
\usepackage{amssymb}
\usepackage[all]{xy}
\usepackage{color}
\usepackage{verbatim}
\usepackage{graphicx}
\usepackage{authblk}
\usepackage{appendix}
\usepackage[numbers]{natbib}
\usepackage{mathrsfs}
\usepackage{perpage}
 \MakePerPage{footnote}

\usepackage{adjustbox}
\usepackage{hyperref}
\hypersetup{colorlinks=true, citecolor=blue, linkcolor=blue}

\usepackage{footnote}

\usepackage{enumitem}

\newtheorem{definition}{Def\text{}inition}[section]
\newtheorem{theorem}[definition]{Theorem}
\newtheorem{example}[definition]{Example}

\newtheorem{lemma}[definition]{Lemma}
\newtheorem{proposition}[definition]{Proposition}
\newtheorem{corollary}[definition]{Corollary}
\newtheorem{remark}[definition]{Remark}
\newtheorem{question}[definition]{Question}

\usepackage{tikz-cd}

\begin{document}

\title{ \bf \large On spaces with star kernel Menger\footnote{The first author was supported by Postdoctoral Fellowship Program at UNAM.}}
\author{ \small JAVIER CASAS-DE LA ROSA, \'ANGEL TAMARIZ MASCAR\'UA}
\date{}
\maketitle

\begin{abstract}
Given a topological property $\mathcal{P}$, a space $X$ is called \emph{star-}$\mathcal{P}$ if for any open cover $\mathcal{U}$ of the space $X$, there exists a set $Y\subseteq X$ with property $\mathcal{P}$ such that $St(Y,\mathcal{U})=X$; the set $Y$ is called a star kernel of the cover $\mathcal{U}$. In this paper, we introduce and study spaces with star kernel Menger, that is, star Menger spaces. Some examples are given to show the relationship with some other related properties studied previously, and the behaviour of the star Menger property with respect to subspaces, products, continuous images and preimages are investigated. Additionally, some comments on the star selection theory are given. Particularly, some questions posed by Song within this theory are addressed. Finally, several new properties are introduced as well as some general questions on them are posed.
 \end{abstract}



\section{Introduction and preliminaries} \label{intro}

Given a property $\mathcal{P}$, a space $X$ is called \emph{star-}$\mathcal{P}$ (or {\it star determined by} $\mathcal{P}$) if for any open cover $\mathcal{U}$ of the space $X$, there is a set $Y\subseteq X$ with the property $\mathcal{P}$ such that $St(Y,\mathcal{U})=X$; the set $Y$ is called star kernel of the cover $\mathcal{U}$. 

The term star-$\mathcal{P}$ was coined in \cite{Mill} but certain star properties were studied before (under different names) by several authors. For example, Ikenaga (\cite{Ikenaga})  studied  star countable, star Lindel\"{o}f and star $\sigma$-compact spaces. On the other hand, the star finite and star countable properties were first studied by van Douwen et.al. (see \cite{DRRT}). A systematic study of some star properties (with different terminology) can be found in \cite{Matveev}.

Fleischman proved in \cite{Fleischman} that the properties of being star finite and countably compact are equivalent in the class of the Hausdorff spaces (see also \cite{DRRT}). On the other hand, Matveev proved in \cite{Matveev}, that pseudocompactness and star pseudocompactness are equivalent. In view of these facts, a special issue of interest for different authors consists in identifying proper classes of pseudocompact spaces that are star-$\mathcal{P}$, for some property $\mathcal{P}$, but are not countably compact (see for example \cite{CMR0}). In that sense, in \cite{Mill} the authors present some examples of spaces that are star-$\mathcal{P}$ for certain compact-like property but are not countably compact. Besides, in \cite{AJW}, the authors present an study and examples of star-$\mathcal{P}$ spaces for Lindel\"{o}f-like properties. Moreover, following the study of star covering properties, many other properties of the type star-$\mathcal{P}$ have been defined using the terminology of \cite{Mill}. For example, in \cite{RT}, the authors introduced and studied the classes of star countable spread spaces and star determined by the countable chain condition as well as, in \cite{RT2}, the class of star countable extent spaces.\\

In order to contribute to this line of investigation, in Section \ref{section on star Menger spaces}, we introduce the class of star Menger spaces and give some examples showing the relationship with other classes of the type star-$\mathcal{P}$. In Section \ref{section on star selection principles theory}, we make some comments on relationships of the star selection principles theory with the star Menger property; a couple of interesting examples are mentioned of which one of them, allows us to address some questions posed by Song that are related to star selection properties. In Section \ref{section on the behaviour of the star Menger property}, we study the behaviour of the star Menger property with respect to subspaces, products, mappings, preimages, etc. In Section \ref{section on equivalences on some classes}, we mention several classes of spaces where the star Menger property is equivalent to some other related properties. Finally, in Section \ref{section on further study and general problems}, we introduce some schemes that provide (potentially) several new properties similar to the one studied in this work; and some general problems about them are posed.

\subsection{Notation and terminology}

Throughout this paper, all spaces are assumed to be regular and $T_1$, unless a specific separation axiom is indicated. For notation and terminology, we refer to \cite{E}.\\
\noindent As usual, $\mathbb{R}$ is the set of the real numbers endowed with its Euclidean topology. We denote by $[X]^{<\omega}$ the collection of all finite subsets of $X$. Given a set $A\subseteq X$ and a family $\mathcal{U}$ of subsets of $X$, the star of $A$ with respect to $\mathcal{U}$, denoted by $St(A,\mathcal{U})$, is the set $\bigcup\{U\in\mathcal{U}:U\cap A\neq\emptyset\}$; for $A=\{x\}$ with $x\in X$, we write $St(x,\mathcal{U})$ instead of $St(\{x\},\mathcal{U})$. Thus, if $A$ and $B$ are subsets of $X$ such that $A\subseteq B$ and $\mathcal{U}$ is a family of subsets of $X$, then $St(A,\mathcal{U})\subseteq St(B,\mathcal{U})$. Also, note that every space which satisfies a property $\mathcal{P}$ is star-$\mathcal{P}$. Moreover, the following basic facts are satisfied by star-$\mathcal{P}$ spaces:
  \begin{enumerate}
    \item[(a)] For any space $X$, if $D$ is a dense subset of $X$ which satisfies a property $\mathcal{P}$, then $X$ is a star-$\mathcal{P}$ space.
    \item[(b)] If $\mathcal{P}$ and $\mathcal{Q}$ are two topological properties such that $\mathcal{P}$ implies $\mathcal{Q}$, then any star-$\mathcal{P}$ space is star-$\mathcal{Q}$.
  \end{enumerate}

A space $X$ is called {\it feebly compact} if every locally finite collection of non-empty open subsets of $X$ is finite, and a space $X$ is called {\it feebly Lindel\"{o}f} if every locally finite family of non-empty open sets in $X$ is countable. Obviously, every feebly compact is feebly Lindel\"{o}f and every feebly compact space is pseudocompact but the converse is not true. However, for Tychonoff spaces the properties of being feebly compact and pseudocompact are equivalent.

For a space $X$, recall that the {\it Alexandroff duplicate of the space} $X$, denoted by $A(X)$, is constructed in the following way: the underlying set of $A(X)$ is $X\times\{0,1\}$ and each point of $X\times\{1\}$ is isolated; a basic open set of a point $(x,0)\in X\times\{0\}$ is a set of the form $(U\times\{0\})\cup((U\times\{1\})\backslash\{(x,1)\})$, where $U$ is an open set in $X$ which contains $x$. It is well-know that $A(X)$ is compact (countably compact, Lindel\"{o}f) if and only if so is $X$, and $A(X)$ is Hausdorff (regular, Tychonoff, normal) if and only if so is $X$.

Recall that a family $\mathcal{A}$ of infinite subsets of $\omega$ is \emph{almost-disjoint} (a.d.) if for any two distinct elements $A,B\in\mathcal{A}$ it holds that $A\cap B$ is finite. Furthermore, an almost-disjoint family is maximal (m.a.d.) if it is not properly contained in another almost-disjoint family. Given an almost-disjoint family $\mathcal{A}$, we define the {\it Mr\'owka-Isbell space determined by} $\mathcal{A}$ (also known as the $\Psi${\it-space generated by} $\mathcal{A}$), which is denoted by $\Psi(\mathcal{A})$, as the space whose underlying set is $\omega\cup\mathcal{A}$ with the following topology: every point in $\omega$ is isolated and, for a given element $A\in\mathcal{A}$, a local base for $A$ is the set $\mathcal{B}(A)=\{\{A\}\cup(A\backslash F):F\in[\omega]^{<\omega}\}$. It is well-known that an almost disjoint family $\mathcal{A}$ is maximal if and only if $\Psi(\mathcal{A})$ is a pseudocompact (see \cite{PM}).

$\sf{CH}$ denotes the Continuum Hypothesis and we refer the reader to \cite{Hodel} to see the definitions of cardinal functions such as spread, extent, cellularity, etc. Recall that the Pixley-Roy topology is given for spaces of subsets of a space $X$ and it is defined as follows: Let $\mathcal{K}$ be a class of subsets of a space $X$. For $K\in\mathcal{K}$ and $U$ an open neighbourhood of $K$ in $X$, let $[K,U]=\{A\in\mathcal{K}: K\subseteq A\subseteq U\}$. Then, the Pixley-Roy topology on $\mathcal{K}$ is obtained by taking the collection $\{[K,U]: K\in\mathcal{K}, U \text{ open in } X\}$ as a base (see \cite{vD}).

\section{Star Menger spaces}\label{section on star Menger spaces}

The following definition was coined in \cite{Mill} by van Mill et.al.

\begin{definition}\label{star-P space}
  Given a topological property $\mathcal{P}$, a space $X$ is called \emph{star-}$\mathcal{P}$ (or {\it star determined by} $\mathcal{P}$) if for any open cover $\mathcal{U}$ of the space $X$, there exists a subspace $Y\subseteq X$ with the property $\mathcal{P}$ such that $St(Y,\mathcal{U})=X$.
\end{definition}

By considering Definition \ref{star-P space}, many authors have studied several classes of star-$\mathcal{P}$ spaces for $\mathcal{P}$ being a compact-like property as well as classes of star-$\mathcal{P}$ spaces for $\mathcal{P}$ being a Lindel\"{o}f-type property (see, for example, \cite{Mill} and \cite{AJW}). Recall that a space is Menger if for every sequence $\{\mathcal{U}_n:n\in\omega\}$ of open covers of $X$, there exists a sequence $\{\mathcal{V}_n:n\in\omega\}$ such that for each $n\in\omega$, $\mathcal{V}_n$ is a finite subcollection of $\mathcal{U}_n$ and $\bigcup\{\mathcal{V}_n:n\in\omega\}$ is an open cover of $X$. Let us introduce one more star-$\mathcal{P}$ class with $\mathcal{P}$ being the Menger property.

\begin{definition}\label{star Menger space}
  We say that a space $X$ is star Menger if for any open cover $\mathcal{U}$ of the space $X$, there exists a Menger subspace $M\subseteq X$ such that $St(M,\mathcal{U})=X$.
\end{definition}

In the following diagram (Figure \ref{relationships among star-P properties}) we summarize implications among star-$\mathcal{P}$ properties (and some other well-known classical properties) that are obtained, either immediately from some basic facts, or by some result from literature. This diagram give us an overview on relationships among these properties with the star Menger property. It is worth to mention that in the class of Tychonoff spaces, pseudocompactness and feebly compactness are equivalent. Hence, for Tychonoff spaces, pseudocompactness trivially implies feebly Lindel\"{o}f.

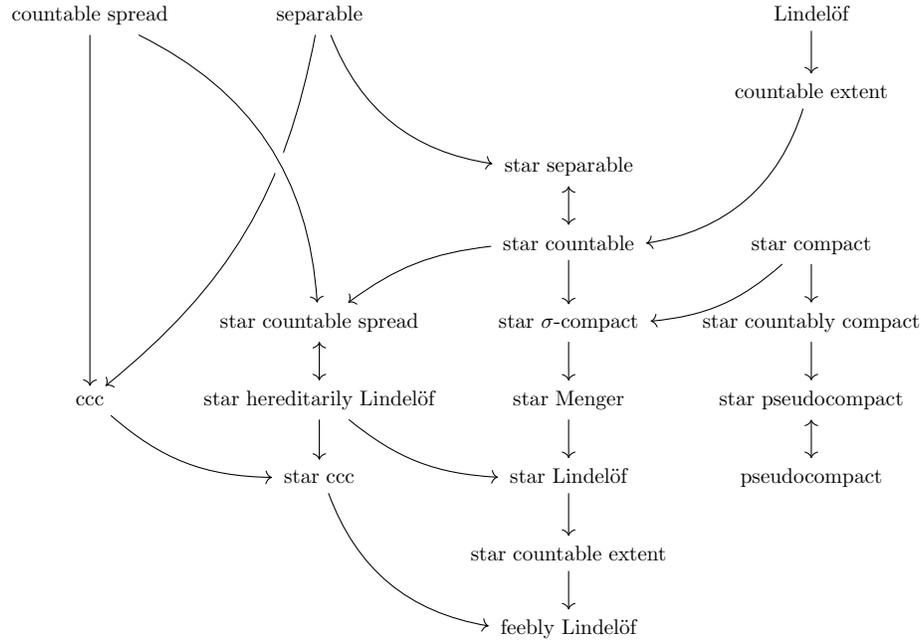
\begin{figure}[h!]
\begin{adjustbox}{center}
\begin{tikzcd}[row sep=1.5em, column sep=.5em]
\scalebox{0.8}{countable spread} \arrow[ddddd] & \scalebox{0.8}{separable} \arrow[dddddl, bend left=20] \arrow[ddr, bend right] & & \scalebox{0.8}{Lindel\"{o}f} \arrow[d] \\ 
& & & \scalebox{0.8}{countable extent} \arrow[ddl, bend left=32] \\ 
& & \scalebox{0.8}{star separable} \arrow[d,leftrightarrow] & \\ 
& & \scalebox{0.8}{star countable} \arrow[dl, bend right=15] \arrow[d] & \scalebox{0.8}{star compact} \arrow[dl, bend left=18] \arrow[d] \\ 
& \scalebox{0.8}{star countable spread} \arrow[d,leftrightarrow] \arrow[uuuul, crossing over, bend right, <-] & \scalebox{0.8}{star $\sigma$-compact} \arrow[d] & \scalebox{0.8}{star countably compact} \arrow[d] \\ 
\scalebox{0.8}{ccc} \arrow[dr, bend right=20] & \scalebox{0.8}{star hereditarily Lindel\"{o}f} \arrow[d] \arrow[dr, bend right=18] & \scalebox{0.8}{star Menger} \arrow[d] & \scalebox{0.8}{star pseudocompact} \arrow[d,leftrightarrow] \\ 
& \scalebox{0.8}{star ccc} \arrow[ddr, bend right=30] & \scalebox{0.8}{star Lindel\"{o}f} \arrow[d] & \scalebox{0.8}{pseudocompact} \\ 
& & \scalebox{0.8}{star countable extent} \arrow[d] & \\ 
& & \scalebox{0.8}{feebly Lindel\"{o}f} & \\ 
\end{tikzcd}
\end{adjustbox}
\caption{Relationships among star-$\mathcal{P}$ properties}
\label{relationships among star-P properties}
\end{figure}

In the following, we will give some examples showing that the star Menger property is different from other star-$\mathcal{P}$ properties defined previously. For that end, we recall the definition of a Luzin space.

\begin{definition}
  We say that a Hausdorff space $X$ is Luzin if it is uncountable and every nowhere dense set of $X$ is countable.
\end{definition}

The following lemma is well-known and will be useful for the example below.

\begin{lemma}[\cite{Hodel}]\label{consequence of celularity}
Let $c(X)\leq\kappa$ and $\mathcal{U}$ be a collection of open subsets of $X$. Then there exists a subcollection $\mathcal{V}$ of $\mathcal{U}$ such that $|\mathcal{V}|\leq\kappa$ and $\bigcup\mathcal{U}\subseteq\overline{\bigcup\mathcal{V}}$.
\end{lemma}

\begin{example}\label{SL not SM}
  Assuming $\sf{CH}$, there exists a star Lindel\"{o}f space which is not star Menger.
\end{example}
\begin{proof}
Given a topological space $(X,\tau)$, we denote by $\mathcal{K}[X]$ the space of all compact nowhere dense subsets of $X$ endowed with the Pixley-Roy topology. In Theorem 3 of \cite{DTW}, it was showed that $\mathcal{K}[\mathbb{R}]$ is a ccc non-separable first countable zero-dimensional Baire space without isolated points. It is easy to see that a similar proof of these conditions for $\mathcal{K}[\mathbb{R}]$ also works for $\mathcal{K}[\mathbb{P}]$, where $\mathbb{P}\subseteq\mathbb{R}$ is the subspace of irrational numbers. Therefore, we can apply Corollary of Theorem 1 \cite{DTW}, to get, under $\sf{CH}$, a dense Luzin subspace $L$ of $\mathcal{K}[\mathbb{P}]$.\\

\noindent \emph{Claim:} The subspace $L$ is Lindel\"{o}f.\\
\noindent Since $\mathcal{K}[\mathbb{P}]$ is a ccc space and $L$ is dense, then $L$ is ccc. Thus, by Lemma \ref{consequence of celularity}, if $\mathcal{U}$ is an open cover of $L$, then there exists a countable subcollection $\mathcal{V}_1$ of $\mathcal{U}$ such that $\bigcup\mathcal{U}\subseteq \overline{\bigcup\mathcal{V}_1}$. It follows that the set $L\setminus\bigcup\mathcal{V}_1$ is nowhere dense in $L$ and therefore, it is countable. Thus, we have a countable subcollection $\mathcal{V}_2$ of $U$ such that $L\setminus\bigcup\mathcal{V}_1\subseteq\bigcup\mathcal{V}_2$. Let $\mathcal{V}=\mathcal{V}_1\cup\mathcal{V}_2$. Then, $\mathcal{V}$ is countable subcover. We conclude that $L$ is Lindel\"{o}f.\\

From this claim, it follows that $\mathcal{K}[\mathbb{P}]$ is star Lindel\"{o}f. Let us show that $\mathcal{K}[\mathbb{P}]$ is not star Menger. We consider the open cover $\mathcal{U}=\{[\{p\} , \mathbb{P}]:p\in\mathbb{P}\}$. We will show that if $\mathcal{M}$ is a Menger subspace of $\mathcal{K}[\mathbb{P}]$, then $St(\mathcal{M},\mathcal{U})\neq\mathcal{K}[\mathbb{P}]$. For that end, let us first show the following fact:\\

\noindent\emph{Fact:} If $\mathcal{M}\subseteq\mathcal{K}[\mathbb{P}]$ is Menger, then $M=\bigcup_{K\in\mathcal{M}}K$ is a Menger subspace of $\mathbb{P}$.\\
Let $(\mathcal{W}_n)_{n\in\omega}$ be a sequence of open covers of $M$. Since for every $K\in\mathcal{M}$, $K\subseteq M$ is compact in $\mathbb{P}$, then for each $n\in\omega$, there exists a finite subcollection $\mathcal{V}_n^K$ of $\mathcal{W}_n$ such that $K\subseteq\bigcup\mathcal{V}_n^K$. Thus, for each $n\in\omega$, $\mathcal{U}_n=\{[K, \bigcup\mathcal{V}_n^K]:K\in\mathcal{M}\}$ is an open cover of $\mathcal{M}$. Now, using the Menger property of $\mathcal{M}$, there exists, for each $n\in \omega$, a finite subcollection $\mathcal{F}_n$ of $\mathcal{U}_n$ such that $\{\bigcup\mathcal{F}_n:n\in\omega\}$ is an open cover of $\mathcal{M}$. We define, for each $n\in\omega$, $\mathcal{K}_n$ as follows: $K\in\mathcal{K}_n$ if and only if $[K, \bigcup\mathcal{V}_n^K]\in\mathcal{F}_n$. Then, we let, for each $n\in\omega$, $\mathcal{H}_n=\{V\in\mathcal{V}_n^K:K\in \mathcal{K}_n\}$. Thus, $\mathcal{H}_n$ is a finite subcollection of $\mathcal{W}_n$ $(n\in\omega)$. Let us show that the collection $\{\bigcup\mathcal{H}_n:n\in\omega\}$ is an open cover of $M$. Let $x\in M$. There exists $K_0\in\mathcal{M}$ such that $x\in K_0$. Then, for such $K_0$, there is $n_0\in\omega$ such that $K_0\in\bigcup\mathcal{F}_{n_0}$. We take $[K, \bigcup\mathcal{V}_{n_0}^K]\in\mathcal{F}_{n_0}$ such that $K_0\in[K, \bigcup\mathcal{V}_{n_0}^K]$. Hence, $K\in\mathcal{K}_{n_0}$ and then $\mathcal{V}_{n_0}^K\subseteq\mathcal{H}_{n_0}$. It follows that $x\in K_{n_0}\subseteq \bigcup\mathcal{V}_{n_0}^K\subseteq\bigcup\mathcal{H}_{n_0}$. This proves that $\{\bigcup\mathcal{H}_n:n\in\omega\}$ is an open cover of $M$. Thus, $M$ is Menger.\\

Now, suppose $\mathcal{M}$ is a Menger subspace in $\mathcal{K}[\mathbb{P}]$. By the previous fact, $M=\bigcup\mathcal{M}$ is Menger in $\mathbb{P}$. It is well know that $\mathbb{P}$ is not Menger. Hence, there is a point $x\in\mathbb{P}\setminus M$. Then $\{x\}\in\mathcal{K}[\mathbb{P}]\setminus St(\mathcal{M}, \mathcal{U})$ since $[\{x\},\mathbb{P}]$ is the only element of $\mathcal{U}$ that contains $\{x\}$ and $[\{x\},\mathbb{P}]\cap\mathcal{M}=\emptyset$ (if there is $C\in[\{x\},\mathbb{P}]\cap\mathcal{M}$, then $\{x\}\subseteq C\subseteq\bigcup\mathcal{M}$ that implies $x\in M$, contradiction).
\end{proof}

It is well known and easy to proof that a space $X$ is star countable if and only if $X$ is star separable. Arguing similarly and using a \v{S}apirovskii theorem saying that every space with spread number at most $\kappa$ has a dense subspace with hereditarily Lindel\"{o}f number at most $\kappa$, the following interesting result was obtained in \cite{RT}.

\begin{theorem}[\cite{RT}]\label{scs equivalent shl}
A space $X$ is star countable spread if and only if $X$ is star hereditarily Lindel\"{o}f.
\end{theorem}

Using Theorem \ref{scs equivalent shl}, an example of a Hausdorff star countable spread that is not star countable was given in \cite{RT}. Additionally, under $\sf{CH}$, a Tychonoff star countable spread non-star countable space was also given in it. Since every star countable is star $\sigma$-compact and every star $\sigma$-compact is star Menger, we can ask the following question requesting for a stronger example.

\begin{question}
Is there an example of a star countable spread space that is not star $\sigma$-compact (star Menger)?
\end{question}

We answer this question by pointing out the following remark.

\begin{remark}\label{L is hL}
  The subspace $L$ of Example \ref{SL not SM} is hereditarily Lindel\"{o}f
\end{remark}
\begin{proof}
Let $Z$ be an uncountable subset of $L$. Note that if $A$ is a nowhere dense subset of $Z$, then $A$ is a nowhere dense subset of $L$. Therefore, any nowhere dense subset of $Z$ is countable. It means $Z$ is Luzin.\\

\noindent\emph{Claim:} Any discrete subset of $Z$ of non-isolated points is nowhere dense.\\
Indeed, let $D$ be a discrete subset of non-isolated points and assume that $U$ a non-empty open set contained in $\overline{D}$. Let $d\in U\cap D$ and take $V$ an open neighbourhood of $d$ such that $V\cap D=\{d\}$. Let $W=U\cap V$. Then $W$ is an open neighbourhood of $d$. Since $d$ is not an isolated point, there is a point $y\in W\setminus\{d\}$. Note that $y$ is an element of $\overline{D}$ and $W\setminus\{d\}$ is an open neighbourhood of it. Therefore, there is a point $d'\in D\setminus\{d\}$ such that $d'\in V$, contradiction. Hence, $D$ is nowhere dense.

Note that $L$ has at most countably many isolated points as it is ccc. It follows that $Z$ has no uncountable discrete subsets. In other words, the spread number of $Z$ is $\omega$, and thus $Z$ has the ccc property. To conclude that $Z$ is Lindel\"{o}f, it is enough to argue as in the proof of the claim of Example \ref{SL not SM} where it was showed that $L$ is Lindel\"{o}f. This proves that $L$ is hereditarily Lindel\"{o}f.
\end{proof}

\begin{example}\label{scs not star Menger}
  Assuming $\sf{CH}$, there exists a Tychonoff star countable spread space that is not star Menger (therefore, not star $\sigma$-compact either).
\end{example}
\begin{proof}
By using Remark \ref{L is hL}, the space $\mathcal{K}[\mathbb{P}]$ is star hereditarily Lindel\"{o}f. Thus, $\mathcal{K}[\mathbb{P}]$ is star countable spread by Theorem \ref{scs equivalent shl}. Since $\mathcal{K}[\mathbb{P}]$ is not star Menger, $\mathcal{K}[\mathbb{P}]$ is not star $\sigma$-compact either.
\end{proof}

Note that as a consequence of Example \ref{scs not star Menger}, we also have that, assuming $\sf{CH}$, there exists a Tychonoff star ccc space that is not star Menger (therefore, not star $\sigma$-compact either). However, we mention that in \cite{RT}, without assuming additional axioms, it was given an example of a Tychonoff star ccc space that is not star countable extent, by using the Pixley-Roy topology. Therefore, we also obtain the following.

\begin{example}\label{star ccc not star Menger}
  There exists a Tychonoff star ccc space that is not star Menger.
\end{example}
\begin{proof}
    Let $X=[\mathbb{R}]^{<\omega}\setminus\{\emptyset\}$ be endowed with the Pixley-Roy topology. Then, $X$ is a ccc Moore space and hence, it is star ccc. Additionally, $X$ is not star Menger as it is not a star countable extent space (see \cite{RT} or \cite{RT2}). 
\end{proof}

On the other hand, using the space given in Proposition 2.6 of \cite{RT}, we also obtain the following example. For the sake of completeness, we mention such a space and we refer the reader to \cite{RT} for details.

\begin{example}\label{star Menger not star ccc}
  There exists a Tychonoff star Menger space that is not star ccc.
\end{example}
\begin{proof}
    Let $\mathcal{A}$ be a mad family on $\omega$ of size $\mathfrak{c}$ and $X=\Psi(\mathcal{A})$. Let $Y=D(\mathfrak{c})\cup\{\infty\}$ be the one-point compactification of the discrete space of size $\mathfrak{c}$. Then, the product space $X\times Y$ is a star $\sigma$-compact (therefore, star Menger) space that is not star ccc (see \cite{RT}). 
\end{proof}

Again, as a consequence of Example \ref{star Menger not star ccc}, we also have the following.

\begin{example}\label{star Menger not star countable spread}
  There exists a Tychonoff star Menger space that is not star countable spread.
\end{example}

We recall that a space $X$ is a $P$-space if all $G_\delta$-sets in $X$ are open. A well-known fact is that a $P$-space is Lindel\"{of} if and only if it is Menger (see \cite{W}). In addition, another well-known fact that is easy to proof is that the product of a Menger space and a compact space is Menger (see \cite{T}).

\begin{example}\label{star Menger not star sigma-compact}
  There exists a Tychonoff star Menger space that is not star $\sigma$-compact.
\end{example}
\begin{proof}
Let $X=[0,\omega_2]$ with the order topology, $Y=D(\omega_1)\cup\{\infty\}$ be the one-point Lindel\"{o}fication of the discrete space $D(\omega_1)=\{d_\alpha:\alpha\in\omega_1\}$ of size $\omega_1$ and $Z=(X\times Y)\setminus\{\langle\omega_2,\infty\rangle\}$. 

Let us show that the space $Z$ is star Menger. Let $\mathcal{U}$ be an open cover of $Z$. For each $\alpha\in\omega_1$, fix $U_\alpha\in\mathcal{U}$ such that $\langle\omega_2, d_\alpha\rangle\in U_\alpha$. Thus, for each $\alpha\in\omega_1$, there exists $\beta(\alpha)\in\omega_2$ such that $(\beta(\alpha), \omega_2]\times\{d_\alpha\}\subseteq U_\alpha$. We define $\beta=sup\{\beta(\alpha):\alpha\in\omega_1\}\in\omega_2$. Then, the subspace $[0, \beta+1]\times Y$ is Menger as $Y$ is a $P$-space. Furthermore, note that $[\beta+1, \omega_2)\times\{\infty\}$ is countably compact and therefore, it is star finite. Thus, for the open cover $\mathcal{U}$, there exists a finite subset $F$ of $[\beta+1, \omega_2)\times\{\infty\}$ such that $[\beta+1, \omega_2)\times\{\infty\}\subseteq St(F,\mathcal{U})$. Let $M=([0, \beta+1]\times Y)\cup F$. Then, $M$ is a Menger subspace of $Z$ so that $St(M,\mathcal{U})=Z$. We conclude that $Z$ is a star Menger space. 

It remains to show that $Z$ is not star $\sigma$-compact. For each $\alpha\in\omega_1$, let $U_\alpha=X\times\{d_\alpha\}$. We define $\mathcal{U}=\{U_\alpha:\alpha\in\omega_1\}\cup\{[0,\omega_2)\times Y\}$. It is clear that $\mathcal{U}$ is an open cover of $Z$. Now, note that if $K_n$ is a compact subspace of $Z$, then $\pi_Y[K_n]$ is compact in $Y$ and thus, a finite subset of $Y$. Hence, if $K$ is a $\sigma$-subspace of $Z$, let say $K=\bigcup_{n\in\omega}K_n$ with each $K_n$ compact, then for each $n\in\omega$, we can fix $\beta_n\in\omega_1$ such that $K_n\cap (X\times\{d_\alpha:\alpha\geq\beta_n\})=\emptyset$. If $\beta=sup\{\beta_n:n\in\omega\}$, then $K\cap(X\times\{d_\alpha:\alpha\geq\beta\})=\emptyset$. Hence, if $\alpha>\beta$, then $\langle\omega_2, d_\alpha\rangle\notin St(K,\mathcal{U})$ as $U_\alpha$ is the only element of $\mathcal{U}$ that contains to $\langle\omega_2, d_\alpha\rangle$ and $U_\alpha\cap K=(X\times\{d_\alpha\})\cap K=\emptyset$. This proves that $Z$ is not star $\sigma$-compact. 
\end{proof}

In \cite{Song0}, the author gave an example of a pseudocompact Tychonoff space that is not star Lindel\"{o}f. We note that this space is even star countably compact. Thus, we obtain the following example.

\begin{example}\label{star cc not star Menger}
  There exists a Tychonoff star countably compact space that is not star Menger.
\end{example}
\begin{proof}
    Let $D=\{d_\alpha:\alpha<\mathfrak{c}\}$ be the discrete space of size $\mathfrak{c}$ and let $$X=(\beta(D)\times (\mathfrak{c}+1))\setminus((\beta(D)\setminus D)\times \{\mathfrak{c}\})$$ be the subspace of $\beta(D)\times (\mathfrak{c}+1)$. Since $\beta(D)\times \mathfrak{c}$ is countably compact and it is dense in $X$, it follows that $X$ is star countably compact. On the other hand, $X$ is not star Menger as it is not even a star Lindel\"{o}f space (see \cite{Song0})
\end{proof}

On the other hand, to give an example of a star Menger that is not star countably compact, it suffices to take any Lindel\"{o}f non-pseudocompact space. In addition, it is also possible to give another example of this kind (star Menger not star countably compact)  but in this case, being pseudocompact not Lindel\"{o}f.

\begin{example}\label{star Menger not star cc}
  There exists a Tychonoff star Menger space that is not star countably compact.
\end{example}
\begin{proof}
    Let $\mathcal{A}$ be a mad family on $\omega$ of size $\mathfrak{c}$ and $X=\Psi(\mathcal{A})$. Then the space $X$ is a pseudocompact space that is not Lindel\"{o}f. Furthermore, since $X$ is separable, $X$ is star Menger. Finally, the space $X$ is not star countably compact (see \cite{CMR0}).
\end{proof}

We finish this section establishing a diagram  (Figure \ref{Negative implications involving the star Menger property}) showing negative implications between the star Menger property and some other star-$\mathcal{P}$ properties obtained as a consequence from examples given in this section; for easy reference in the diagram, we summarize such examples in the following list:

\begin{enumerate}
    \item \label{first} Example \ref{SL not SM} is a star Lindel\"{o}f space that is not star Menger.
    \item \label{second} Example \ref{scs not star Menger} is a star countable spread space that is not star Menger.
    \item \label{third} Example \ref{star ccc not star Menger} is a star ccc space that is not star Menger.
    \item \label{fourth} Example \ref{star Menger not star ccc} is a star Menger space that is not star ccc.
    \item \label{fifth} Example \ref{star Menger not star countable spread} is a star Menger space that is not star countable spread.
    \item \label{sixth} Example \ref{star Menger not star sigma-compact} is a star Menger space that is not star $\sigma$-compact.
    \item \label{seventh} Example \ref{star cc not star Menger} is a star countably compact space that is not star Menger.
    \item \label{eighth} Example \ref{star Menger not star cc} is a star Menger space that is not star countably compact.
\end{enumerate}

\begin{figure}[h!]
\begin{adjustbox}{center}
\begin{tikzcd}[row sep=1.5em, column sep=.5em]
\scalebox{0.8}{countable spread} \arrow[ddddd] & \scalebox{0.8}{separable} \arrow[dddddl, bend left=20] \arrow[ddr, bend right] & & \scalebox{0.8}{Lindel\"{o}f} \arrow[d] \\ 
& & & \scalebox{0.8}{countable extent} \arrow[ddl, bend left=32] \\ 
& & \scalebox{0.8}{star separable} \arrow[d,leftrightarrow] & \\ 
& & \scalebox{0.8}{star countable} \arrow[dl, bend right=15] \arrow[d] & \scalebox{0.8}{star compact} \arrow[dl, bend left=18] \arrow[d] \\ 
& \scalebox{0.8}{star countable spread} \arrow[d,leftrightarrow] \arrow[uuuul, crossing over, bend right, <-] \arrow[dr, "\scalebox{0.6}{/}" marking, "\ref{second}", bend left=10, blue] & \scalebox{0.8}{star $\sigma$-compact} \arrow[d] & \scalebox{0.8}{star countably compact} \arrow[d] \arrow[dl, crossing over, "\scalebox{0.6}{/}" marking, "\ref{seventh}", bend left=8, blue] \arrow[dl, crossing over, "\scalebox{0.6}{/}" marking, "\ref{eighth}"', bend right=5, <-, blue] \\ 
\scalebox{0.8}{ccc} \arrow[dr, bend right=20] & \scalebox{0.8}{star hereditarily Lindel\"{o}f} \arrow[d] \arrow[dr, crossing over, bend right=18] & \scalebox{0.8}{star Menger} \arrow[d] \arrow[ul, "\scalebox{0.6}{/}" marking, "\ref{fifth}"', bend left=10, blue] \arrow[u, "\scalebox{0.6}{/}" marking, shift right=1.5ex, "\;\ref{sixth}"', blue] & \scalebox{0.8}{star pseudocompact} \arrow[d,leftrightarrow] \\ 
& \scalebox{0.8}{star ccc} \arrow[ddr, bend right=30] \arrow[ur, crossing over, "\scalebox{0.6}{/}" marking, "\ref{third}", bend left=8, blue] \arrow[ur, crossing over, "\scalebox{0.6}{/}" marking, "\ref{fourth}"', bend right=8, <-, blue] & \scalebox{0.8}{star Lindel\"{o}f} \arrow[d] \arrow[u, "\scalebox{0.6}{/}" marking, shift right=1.5ex, "\;\ref{first}"', blue] & \scalebox{0.8}{pseudocompact} \\ 
& & \scalebox{0.8}{star countable extent} \arrow[d] & \\ 
& & \scalebox{0.8}{feebly Lindel\"{o}f} & \\ 
\end{tikzcd}
\end{adjustbox}
\label{Negative implications involving the star Menger property}
\caption{Negative implications involving the star Menger property}
\end{figure}
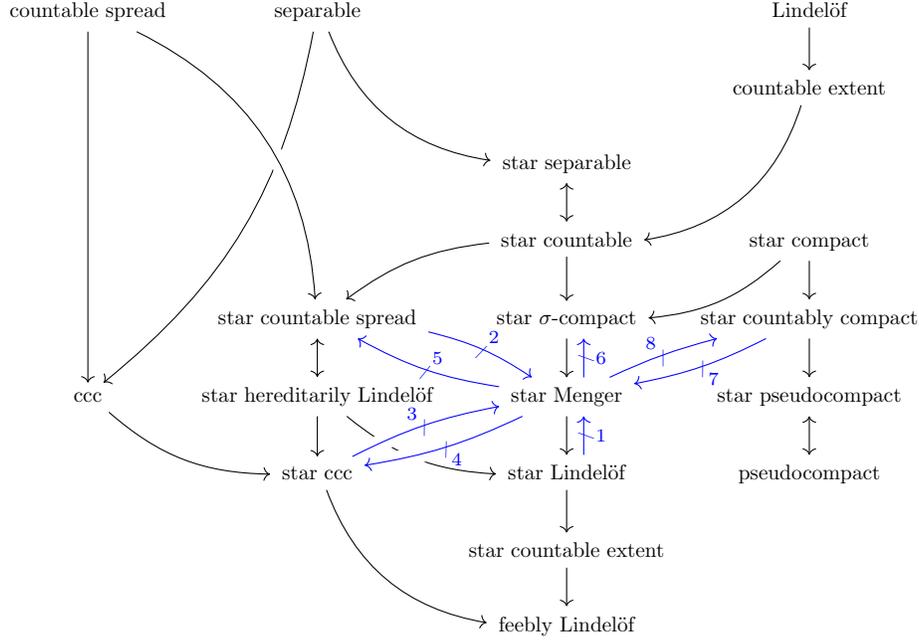

\section{Comments on star selection principles theory}\label{section on star selection principles theory}

In this section, we point out some remarks about the relationship between some star selection principles with the star Menger property and, we also give some examples showing the star Menger property does not coincide with related properties defined in this theory. Further, we address some questions from \cite{SongL}, \cite{SongK} and \cite{SongRK}.

We begin mentioning some star selection principles introduced by Ko\v{c}inac in \cite{K} and \cite{K2}. Before of that, it is important pointing that a property belonging to this theory is called the same (star-Menger property) as the property studied in this work. However, in order to avoid notational confusion, we will denote it just as $SM$ property. Let $\mathscr{A}$ and $\mathscr{B}$ be collections of families of subsets of $X$ and $\mathcal{K}$ be a family of subsets of $X$:\\
\newline
$\mathbf{S^*_{fin}(\mathscr{A},\mathscr{B})}$: For any sequence $\{\mathcal{A}_n:n\in\omega\}$ of elements of $\mathscr{A}$, there is a sequence $\{\mathcal{B}_n:n\in\omega\}$ such that $\mathcal{B}_n$ is a finite subset of $\mathcal{A}_n$, $n\in\omega$, and $\bigcup_{n\in\omega}\{St(B,\mathcal{A}_n):B\in\mathcal{B}_n\}\in\mathscr{B}$.\\
\newline
$\mathbf{SS^*_{\mathcal{K}}(\mathscr{A},\mathscr{B})}$: For any sequence $\{\mathcal{A}_n:n\in\omega\}$ of elements of $\mathscr{A}$, there is a sequence $\{K_n:n\in\omega\}$ of elements of $\mathcal{K}$ such that $\{St(K_n,\mathcal{A}_n):n\in\omega\}\in\mathscr{B}$.\\

When $\mathcal{K}$ is the collection of all finite (resp. compact) subsets of $X$, it is denoted by $\mathbf{SS^*_{fin}(\mathscr{A},\mathscr{B})}$ (resp. $\mathbf{SS^*_{comp}(\mathscr{A},\mathscr{B})}$) instead of $\mathbf{SS^*_{\mathcal{K}}(\mathscr{A},\mathscr{B})}$. Given a topological space $X$, if $\mathscr{O}$ denote the collection of all open covers of $X$, then $S^*_{fin}(\mathscr{O},\mathscr{O})$ defines the star-Menger property ($SM$), $SS^*_{fin}(\mathscr{O},\mathscr{O})$ defines the strongly star-Menger property ($SSM$). Also, $SS^*_{comp}(\mathscr{O},\mathscr{O})$ defines the star-K-Menger property ($S$-$K$-$M$). Explicitly, a space $X$ is strongly star-Menger ($SSM$) if for each sequence $\{\mathcal{U}_n:n\in\omega\}$ of open covers of $X$, there exists a sequence $\{F_n:n\in\omega\}$ of finite subsets of $X$ such that $\{St(F_n, \mathcal{U}_n):n\in\omega\}$ is an open cover of $X$; a space $X$ is star-Menger ($SM$) if for each sequence $\{\mathcal{U}_n:n\in\omega\}$ of open covers of $X$, there exists a sequence $\{\mathcal{V}_n:n\in\omega\}$ such that $\mathcal{V}_n$ is a finite subset of $\mathcal{U}_n$ for each $n\in\omega$, and $\bigcup_{n\in\omega}\{St(V, \mathcal{U}_n):V\in \mathcal{V}_n\}$ is an open cover of $X$; a space $X$ is star-K-Menger ($S$-$K$-$M$) if for every sequence $\{\mathcal{U}_n:n\in\omega\}$ of open covers of $X$, there exists a sequence $\{K_n:n\in\omega\}$ of compact subsets of $X$ such that $\{St(K_n,\mathcal{U}_n):n\in\omega\}$ is an open cover of $X$.
We refer the reader to \cite{K_survey} to see the current status of previous properties and some other star selection principles. One more property belonging to this theory that has been extensively studied by several authors under different terminology is the star-Lindel\"{o}f property (see \cite{DRRT}) which is defined as follows: A space $X$ is star-Lindel\"{o}f, briefly $SL$, if for every open cover $\mathcal{U}$ of $X$ there exists a countable subcollection $\mathcal{V}$ of $\mathcal{U}$ such that $St(\bigcup\mathcal{V},\mathcal{U})=X$; once again, to avoid notational confusion, we will denote it just by $SL$ property.

Following the general star selection hypothesis  $\mathbf{SS^*_{\mathcal{K}}(\mathscr{A},\mathscr{B})}$, we may also consider the class of star-$M$-Menger in the same way as the classes of star-$K$-Menger and star-$C$-Menger have been defined and studied (see \cite{SongK} y \cite{SY}). Namely, 

\begin{definition}\label{star-M-Menger space}
  A space $X$ is said to be star-$M$-Menger if for any sequence $\{\mathcal{U}_n:n\in\omega\}$ of open covers of $X$, there is a sequence $\{M_n:n\in\omega\}$ of Menger subsets of $X$ such that $\{St(M_n,\mathcal{U}_n):n\in\omega\}$ is an open cover of $X$.
\end{definition}

And making a slight modification to the previous definition, we can establish the following notion:

\begin{definition}\label{star-M_f-Menger space}
  A space $X$ is said to be star-$M_f$-Menger if for any sequence $\{\mathcal{U}_n:n\in\omega\}$ of open covers of $X$, there is a Menger subspace $M\subseteq X$ such that $\{St(M,\mathcal{U}_n):n\in\omega\}$ is an open cover of $X$.
\end{definition}

Note that Definitions \ref{star-M-Menger space} and \ref{star-M_f-Menger space} can be viewed as new star selection properties. However, the following remark, easy to prove, says that these two notions coincide with the star Menger property given in Definition \ref{star Menger space}.

\begin{remark}
Let $X$ be a topological space. Then, the following are equivalent:
\begin{enumerate}
    \item[(i)] $X$ satisfies the star Menger property;
    \item[(ii)] $X$ satisfies the star-$M_f$-Menger property;
    \item[(iii)] $X$ satisfies the star-$M$-Menger property.
\end{enumerate}
\end{remark}
\begin{proof}
To show $(i)\rightarrow(ii)$, let $(\mathcal{U}_n)_{n\in\omega}$ be a sequence of open covers of $X$. Applying the star Menger property of $X$ for the open cover $\mathcal{U}_0$, we can take a Menger subspace of $X$ such that $St(M,\mathcal{U}_0)=X$. It trivially follows that the collection $\{St(M,\mathcal{U}_n):n\in\omega\}$ is an open cover of $X$. Hence, $X$ satisfies the star-$M_f$-Menger property.

To show $(ii)\rightarrow(iii)$, let $(\mathcal{U}_n)_{n\in\omega}$ be a sequence of open covers of $X$. Then, by hypothesis, there exists a Menger subspace of $X$ such that the collection $\{St(M,\mathcal{U}_n):n\in\omega\}$ is an open cover of $X$. Defining for each $n\in\omega$, $M_n=M$, we obtain that the collection $\{St(M_n,\mathcal{U}_n):n\in\omega\}$ is an open cover of $X$. Thus, $X$ satisfies the star-$M$-Menger property.

To show $(iii)\rightarrow(i)$, let $\mathcal{U}$ be an open cover of $X$. We put, for each $n\in\omega$, $\mathcal{U}_n=\mathcal{U}$. Then, applying the star-$M$-Menger property of $X$ to the sequence $(\mathcal{U}_n)_{n\in\omega}$, there exists a sequence $(M_n)_{n\in\omega}$ of Menger subspaces of $X$ such that the collection $\{St(M_n,\mathcal{U}_n):n\in\omega\}$ is an open cover of $X$. We define $M=\bigcup_{n\in\omega}M_n$. Then $M$ is Menger. Furthermore, for each $n\in\omega$, we have $St(M_n, \mathcal{U}_n)\subseteq St(M, \mathcal{U}_n)$. It follows that the collection $\{St(M,\mathcal{U}_n):n\in\omega\}$ is an open cover of $X$. Note that, for each $n\in \omega$, $St(M,\mathcal{U}_n)=St(M,\mathcal{U})$. Hence, $X=St(M,\mathcal{U})$. It shows that $X$ satisfies the star Menger property.
\end{proof}

The previous observation allows us to relate some star selection principles with some star-$\mathcal{P}$ properties. Next, we establish a diagram (Figure \ref{star selection principles and star-P properties}) that shows the immediate implications among star selection properties and related star-$\mathcal{P}$ properties.

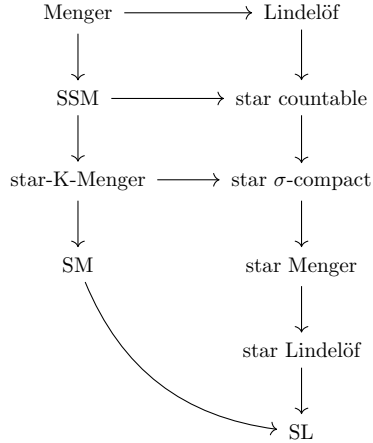
\begin{figure}[h!]
\begin{adjustbox}{center}
\begin{tikzcd}
\scalebox{0.8}{Menger}  \arrow[r] \arrow[d] & \scalebox{0.8}{Lindel\"{o}f} \arrow[d ]\\
\scalebox{0.8}{SSM} \arrow[r] \arrow[d]  & \scalebox{0.8}{star countable} \arrow[d]  \\
\scalebox{0.8}{star-K-Menger} \arrow[r] \arrow[d] & \scalebox{0.8}{star $\sigma$-compact} \arrow[d]\\
\scalebox{0.8}{SM} \arrow[ddr, bend right] & \scalebox{0.8}{star Menger} \arrow[d]\\
  & \scalebox{0.8}{star Lindel\"{o}f} \arrow[d]\\
  & \scalebox{0.8}{SL}
\end{tikzcd}
\end{adjustbox}
\caption{Some star selection principles and star-$\mathcal{P}$ properties}
\label{star selection principles and star-P properties}
\end{figure}

As in Section \ref{section on star Menger spaces}, we may wonder if the star Menger property coincides with some of the properties from star selection theory. In particular, we wonder if there is a relationship between the $SM$ property and the star Menger property. To answer it, we present next examples.

Recall that the Niemytzki plane\index{Niemytzki plane $N(X)$} on a set $X\subseteq\mathbb{R}$, denoted by $N(X)$, has as underlying set $X\times\{0\}\cup\mathbb{R}\times(0,\infty)$. The open upper half-plane $\mathbb{R}\times(0,\infty)$ has the Euclidean topology and the set $X\times\{0\}$ has the topology generated by all sets of the form $\{(x,0)\}\cup B$ where $x\in X$ and $B$ is an open disc in $\mathbb{R}\times(0,\infty)$ which is tangent to $X\times\{0\}$ at the point $(x,0)$. It is well-known that $N(\mathbb{R})$ is completely regular, separable, but not normal, not Lindel\"{o}f, not second countable, not paracompact.

\begin{example}\label{star Menger not SM}
  There exists a Tychonoff star Menger space that is not $SM$. In addition, there is a model where do exist a normal star Menger space that is not $SM$.
\end{example}
\begin{proof}
Since $N(\mathbb{R})$ is separable, it is a star Menger space. In \cite{CGS}, it was pointed out that $N(X)$ is not $SM$ for any $X\subseteq\mathbb{R}$ of size $\mathfrak{c}$. Thus, in particular, $N(\mathbb{R})$ is not $SM$. On the other hand, recall that a set $X\subseteq\mathbb{R}$ is a $Q$-set if for every $A\subseteq X$, $A$ is $F_\sigma$ in $X$ and $|X|>\aleph_0$. It is known that the existence of $Q$-sets are consistent with (and independent from) $\sf{ZFC}$ (see for instance \cite{Miller}). Furthermore, a Niemytzki plane $N(X)$ is normal if and only if $X$ is a $Q$-set (see \cite{Tall}). Again, in \cite{CGS}, it is pointed out that in \cite{JS}, Judah and Shelah built a model with a $Q$-set of size $\mathfrak{d}$. Therefore, in such a model, $N(X)$ with $X$ being such a $Q$-set of size $\mathfrak{d}$ is an example of a normal star Menger that is not $SM$ by Proposition 2.12 in \cite{S}.
\end{proof}

On other hand, we also have an example showing that the $SM$ property does not imply the star Menger property. In Example E of \cite{Tall}, Tall presented, assuming $2^{\aleph_0}=2^{\aleph_1}$, an example of a normal separable space with an uncountable closed discrete set. We provide details of the construction of such example for sake of completeness:\\

\noindent {\bf Construction:} Let $L$ be a set of cardinality $\aleph_1$ disjoint from $\omega$ and $\mathcal{F}$ be a strongly independent family of subsets of $\omega$ of size $2^{\aleph_0} = \mathfrak{c}$.\\
We write $\mathcal{F} = \{A_\alpha: \alpha < \mathfrak{c}\}$. Since $|L| =\aleph_1$, $|\mathcal{P}(L)| = 2^{\aleph_1}$. Since we are assuming $2^{\aleph_0} = 2^{\aleph_1}$, it is possible to build a function $f:\mathcal{P}(L) \to \{A_\alpha: \alpha < \mathfrak{c}\} \cup \{\omega \smallsetminus A_\alpha: \alpha < \mathfrak{c}\}$ which is bijective and complement-preserving (for each $B \subseteq L$, $f(L\smallsetminus B) = \omega \smallsetminus f(B)$).\\
Now, let $X_0 = L \cup \omega$ with a subbase $\varphi$ for a topology defined by
\begin{enumerate}
\item if $M\subseteq L$, then $M \cup f(M) \in \varphi$,
\item if $n \in \omega$, then $\{n\} \in \varphi$,
\item if $p \in X$, then $X \smallsetminus \{p\} \in \varphi$.
\end{enumerate}
Then $X_0$ is a normal $T_1$ separable space with an uncountable closed discrete subspace.
\vspace{.5cm}

In \cite{SongROC}, Song defined a modification of Tall's example to get an example, under $2^{\aleph_0}=2^{\aleph_1}$, of a normal feebly Lindel\"{of} space that is not star Lindel\"{o}f\footnote{Song gave this example to partially answer a question posed in \cite{AJW}.} and very recently, in \cite{CCG}, the authors used Song's example to get an example, assuming $2^{\aleph_0}=2^{\aleph_1}$ and $\omega_1<\mathfrak{d}$, of a normal $SM$ space that is not $SSM$ (not Dowker space)\footnote{The authors of \cite{CCG} gave this example to answer a couple of questions posed in \cite{CGS}}. In conclusion, we have the following

\begin{example}\label{normal SM not star Menger}
    Assuming $2^{\aleph_0}=2^{\aleph_1}$ and $\omega_1<\mathfrak{d}$, there exists a normal $SM$ space that is not star Menger.
\end{example}
\begin{proof}
    Let $X_0 = L \cup \omega$ denote the space given above. Let $X = L \cup (\omega_1 \times \omega)$ with the topology given as follows: a basic open set of
\begin{description}
\item[(i)] $x \in L$ is a set of the form $V^U_\alpha(x) = (U \cap L) \cup \big((\alpha, \omega_1) \times (U \cap \omega)\big)$ where $U$ is a neighbourhood of $x \in X_0$ and $\alpha < \omega_1$.
\item[(ii)] $\langle \alpha , n \rangle \in (\omega_1 \times \omega)$ is a set of the form $V_W(\langle \alpha , n \rangle) = W \times \{n\}$ where $W$ is a neighbourhood of $\alpha$ in $\omega_1$ with the usual topology.
\end{description}
Then $X$ is a normal $SM$ space (see \cite{CCG} for details) that is not star Lindel\"{of} (see \cite{SongROC}) and therefore neither star Menger.
\end{proof}

By Theorem 2.7 in \cite{AJW}, it is known that every star Lindel\"{o}f space is feebly Lindel\"{o}f and, it is easy to prove that every star Lindel\"{o}f space is also $SL$. However, the reverse of those implications does not hold in general; an example of a feebly Lindel\"{o}f not star Lindel\"{o}f space was given in \cite{AJW} and, an example of a $SL$ space that is not star Lindel\"{o}f can be found in \cite{SongL} (under different terminology). Nevertheless, one might wonder if both properties together imply the star Lindel\"{o}f property, that is to say, if a feebly Lindel\"{o}f $SL$ space is star Lindel\"{o}f; it turns out that it is not true in general. In \cite{SongRON}, Song proved that the space $X$ described in Example \ref{normal SM not star Menger} is $SL$. Thus, such a space is an example of a feebly Lindel\"{o}f $SL$ space that is not star Lindel\"{o}f.\\

The following diagram (Figure \ref{Negative implications with star selection principles}) includes the negative implications between the star Menger property and the $SM$ property obtained as consequences from two examples given in this section; such examples are in the following list:

\begin{enumerate}
    \item \label{SSP:first} Example \ref{star Menger not SM} is a star Menger space that is not $SM$.
    \item \label{SSP:second} Example \ref{normal SM not star Menger} is a $SM$ space that is not star Menger.
\end{enumerate}

\begin{figure}[h!]
\begin{adjustbox}{center}
\begin{tikzcd}
\scalebox{0.8}{Menger}  \arrow[r] \arrow[d] & \scalebox{0.8}{Lindel\"{o}f} \arrow[d ]\\
\scalebox{0.8}{SSM} \arrow[r] \arrow[d]  & \scalebox{0.8}{star countable} \arrow[d]  \\
\scalebox{0.8}{star-K-Menger} \arrow[r] \arrow[d] & \scalebox{0.8}{star $\sigma$-compact} \arrow[d] \\
\scalebox{0.8}{SM} \arrow[ddr, bend right] \arrow[r, "\scalebox{0.6}{/}" marking, "\;\;\;\ref{SSP:second}", shift left=.8ex, blue] \arrow[r, "\scalebox{0.6}{/}" marking, "\;\;\;\ref{SSP:first}"', shift right=.8ex, <-, blue]  & \scalebox{0.8}{star Menger} \arrow[d] \\
  & \scalebox{0.8}{star Lindel\"{o}f} \arrow[d]\\
  & \scalebox{0.8}{SL}
\end{tikzcd}
\end{adjustbox}
\caption{Negative implications with star selection principles}
\label{Negative implications with star selection principles}
\end{figure}
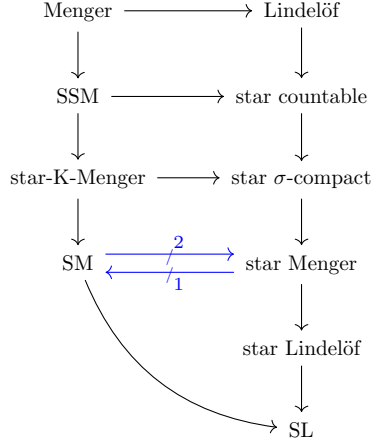

To finish this section, we address some questions due to Song. In \cite{SongL}, Song gave an example (under different terminology) of a Tychonoff $SL$ space that is not star Lindel\"{o}f and he asked (see Remark 1 in \cite{SongL}) if a normal $SL$ space is star Lindel\"{o}f; this is not the case. The space in Example \ref{normal SM not star Menger} is a (consistent) normal $SL$ space (\cite{SongRON}) that is not star Lindel\"{o}f (\cite{SongROC}). On the other hand, in \cite{SongK}, Song gave an example of a $T_1$ $SM$ space that is not star-$K$-Menger\footnote{There is a mistake in this example; it is not difficult to show that the space in this example is, in fact, star-$K$-Menger.} and he asked (see Remark 2.5 in \cite{SongK}) about a Hausdorff (or Tychonoff) example. Thereafter, in \cite{SongRK}, the author gave a Hausdorff example and asked, again, about a regular or Tychonoff example (see Question 2.1 in \cite{SongRK}), that is, is there a regular (or Tychonoff) $SM$ space which is not star-$K$-Menger? To answer it, it is easy to show that if a space $X$ is star-$K$-Menger, then $X$ is star $\sigma$-compact (and therefore, star Lindel\"{o}f). Thus, Example \ref{normal SM not star Menger} consistently answer this question in the affirmative, since that example is a normal $SM$ space that is not star Lindel\"{o}f (hence, neither star-$K$-Menger).

\section{Behaviour of the star Menger property}\label{section on the behaviour of the star Menger property}

In this section we study how the star Menger property behaves. In particular, we investigate the type of subspaces to which it is inherited, when the product of two star Menger spaces conserve this property, and under which conditions on continuous functions it is a direct or inverse invariant.

Although it is easily showed that the Menger property is inherited by closed subspaces, the star Menger property does not behave the same. If $X=\Psi(\mathcal{A})$ with $\mathcal{A}$ being a mad family on $\omega$, then $X$ is a star Menger space (since this space is separable). However, it is clear that the subset $\mathcal{A}$ is not star Menger. Since this subset is closed $G_\delta$ (even it is a zero set), we conclude:

\begin{proposition}
The star Menger property is not necessarily inherited either by closed subset or closed $G_\delta$-sets or zero sets.  
\end{proposition}

On the other hand, if we consider $X=L(D(\omega_1))$ being the one-point Lindel\"{o}fication of the discrete space $D(\omega_1)$ of size $\omega_1$, then $X$ is a star Menger space and $D(\omega_1)$ is an open dense subset which is not star Menger. Therefore, we obtain:

\begin{proposition}
  The star Menger property is not necessarily inherited either by open sets or dense sets, nor even by open dense sets.
\end{proposition}

The following result is showed in \cite{RT},

\begin{proposition}[\cite{RT}]
  Let $\mathcal{P}$ be a property that is preserved by countable unions. If $\mathcal{P}$ is inherited by open subsets or if $\mathcal{P}$ is inherited by closed subsets, then the star $\mathcal{P}$ property is inherited by open $F_\sigma$-subsets.
\end{proposition}

It is well-known that the Menger property is preserved by countable unions and it is inherited by closed subsets. Hence we have the following facts,

\begin{proposition}\label{On clopen sets}
  The star Menger property is inherited by open $F_\sigma$-subsets. In particular, the star Menger property is inherited by clopen sets as well as by cozero sets.
\end{proposition}

Another kind of subsets for which it is worth asking if the star Menger property is inherited by, are the regular closed subsets and the regular open subsets. In contrast to the above, the behaviour of star Menger property is not good even for these subsets.

\begin{proposition}
The star Menger property is not necessarily inherited by regular closed subsets.  
\end{proposition}
\begin{proof}
Let $\mathcal{A}_1$ be an almost disjoint family on $\omega$ with $|\mathcal{A}_1|=\omega_1$. We consider $X_1=\mathcal{A}_1\cup(\omega_1\times\omega)$ with the topology generated by the following basic open sets: The set $\omega_1\times\omega$ has the usual product topology (where $\omega_1$ is considered with the order topology) and it is an open subset of $X_1$. For each $a\in\mathcal{A}_1$, a basic open set has the form $B_{\beta, F}(a)=[(\beta,\omega_1)\times(a\setminus F)]\cup\{a\}$ with $\beta\in\omega_1$ and $F$ being a finite subset of $a$.

Let us show that $X_1$ is not star Menger. We enumerate $\mathcal{A}_1=\{a_\alpha:\alpha\in\omega_1\}$. Then, for each $\alpha\in\omega_1$, let $U_\alpha=[(\alpha, \omega_1)\times a_\alpha]\cup\{a_\alpha\}$. Thus, the collection $\mathcal{U}=\{U_\alpha:\alpha\in\omega_1\}\cup\{\omega_1\times\omega\}$ is an open cover of $X_1$. Let $M$ be any Menger subspace of $X_1$. Since $\mathcal{A}_1$ is a closed discrete subset of $X_1$, then  $M\cap\mathcal{A}_1$ is at most countable. So, we can take $\beta_1\in\omega_1$ such that for each $\alpha>\beta_1$, $a_\alpha\notin M$. On the other hand, note that for each $n\in\omega$ the set $\omega_1\times\{n\}$ is a closed subset of $X_1$. Thus, $(\omega_1\times\{n\})\cap M$ is a Menger subset of $\omega_1\times\{n\}$ and then there exists $\gamma_n\in\omega_1$ such that $[(\gamma_n,\omega_1)\times\{n\}]\cap M=\emptyset$. Let $\beta_2=sup\{\gamma_n:n\in\omega\}$. Then, we define $\beta=max\{\beta_1,\beta_2\}$. It follows that $U_\alpha\cap M=\emptyset$ for any $\alpha>\beta$. Since $U_\alpha$ is the only element of $\mathcal{U}$ containing to $a_\alpha$, then $a_\alpha\notin St(M,\mathcal{U})$ with $\alpha>\beta$. We conclude that $X_1$ is not star Menger.

Now, let $X_2=\Psi(\mathcal{A}_2)$ with $\mathcal{A}_2$ being an almost disjoint family on $\omega$ of size $\omega_1$. Note that $X_2$ is star Menger as it is separable. Assume $X_1\cap X_2=\emptyset$. We take a bijection $f:\mathcal{A}_1\rightarrow \mathcal{A}_2$ and let $X$ be the quotient space obtained from the discrete sum $X_1\oplus X_2$ by identifying $a_\alpha$ of $\mathcal{A}_1$ with $f(a_\alpha)$ of $\mathcal{A}_2$ for each $\alpha\in\omega_1$. Namely, $$X=(\omega_1\times\omega)\cup\{(a_\alpha, f(a_\alpha)): \alpha\in\omega_1\}\cup\omega.$$
Let $q:X_1\oplus X_2\rightarrow X$ be the quotient map and we consider $Y=q[X_1]$. Since $Y=\overline{\omega_1\times\omega}$, it follows that $Y$ is a regular closed subset in $X$ . However, $Y$ is not star Menger as it is homeomorphic to $X_1$. Indeed, $q\!\!\restriction_{X_1}$ is a continuous bijection. To see it is open, let $a_\alpha\in\mathcal{A}_1$ and $F$ be a finite subset of $a_\alpha$. Then, $q[\{a_\alpha\}\cup((\beta,\omega_1)\times a_\alpha\setminus F)]=U\cap Y$ with $U=\{(a_\alpha, f(a_\alpha))\}\cup((\beta,\omega_1)\times(a_\alpha\setminus F))\cup f(a_\alpha)$ which is an open set in $X$. So, $U\cap Y$ is an open set in $Y$ and then  $q\!\!\restriction_{X_1}$ is open. It shows  $q\!\!\restriction_{X_1}$ is an homeomorphism. Finally, let us show that $X$ is star Menger. Let $\mathcal{U}$ be an open cover of $X$. Since for each $n\in\omega$, $\omega_1\times\{n\}$ is countably compact, there is a finite subset $F_n$ of $q[\omega_1\times\{n\}]$ such that $q[\omega_1\times\{n\}]\subseteq St(F_n,\mathcal{U})$. We put $C_1=\bigcup_{n\in\omega}F_n$. Thus, $q[\omega_1\times\omega]\subseteq St(C_1,\mathcal{U})$. On the other hand, we can take a countable subset $C_2$ of $q[X_2]$ such that $q[X_2]\subseteq St(C_2,\mathcal{U})$ as $q[X_2]$ is separable. Let $C=C_1\cup C_2$. It follows that $X = St(C,\mathcal{U})$. This shows that $X$ is even star countable and hence, star Menger.
\end{proof}

We have same situation for regular open subsets. In fact, there is an example of a normal star Menger space with a regular open subset that is not star Menger.

\begin{proposition}
The star Menger property is not necessarily inherited by regular open subsets.  
\end{proposition}
\begin{proof}
    Let $L=D(\omega_1)\cup\{\infty\}$ be the one-point Lindelolification of the discrete space $D(\omega_1)$ of size $\omega_1$. Then, $L$ is a star Menger space. Put $D(\omega_1)=\{d_\alpha:\alpha\in\omega_1\}$ and take a partition of $\omega_1=A\cup B$ with $|A|=|B|=\omega_1$. Let $D(A)=\{d_\alpha:\alpha\in A\}$. Note that $D(A)$ is a regular open subset of $L$ and it is not star Menger as it is an uncountable discrete subset of $L$.
\end{proof}

Now, regarding products, it turns out that the star Menger property need not be preserved by products, not even by finite products. In fact, we will see that there are spaces with stronger properties than star Menger property whose product is not star Menger. We start by considering countably compactness. The construction of these spaces is standard and well-known. For convenience of the reader, we recall such a construction.

\begin{proposition}\label{products with cc}
     There are two countably compact spaces (and hence, star Menger spaces) whose product is not star Menger.
\end{proposition}
\begin{proof}
     Let us define two subspaces $X, Y$ of $\beta(D)$, with $D$ being the discrete space of size $\mathfrak{c}$, such that $X\cap Y=D$, $X\cup Y=\beta(D)$ and $X, Y$ are countably compact. Let $X_0=D$. By transfinite induction, it is easy to define, for each $\alpha\in\omega_1$, $$X_\alpha=\bigcup_{\gamma<\alpha}X_\gamma\cup f\left[ \left[ \bigcup_{\gamma<\alpha}X_\gamma \right]^\omega\right]$$ where $f:\left[\beta(D)\right]^\omega\rightarrow\beta(D)$ is a function assigning to each infinite countable subset $A$ of $\beta(D)$ a limit point of $A$ in $\beta(D)$. Let $X=\bigcup_{\alpha\in\omega_1}X_\alpha$ and $Y=D\cup[\beta(D)\setminus X]$. Thus, $X\cap Y=D$ and $X\cup Y=\beta(D)$. In addition, by construction, it is easy to show that $X$ and $Y$ are countably compact and then, star Menger spaces. However, $X\times Y$ is not star Menger since it contains the discrete clopen set $\Delta=\{(x,x):x\in D\}$ that is not star Menger.
\end{proof}

The spaces $X$ and $Y$ used in Proposition \ref{products with cc} show that not even the product of countably compact spaces need not be star Menger. Next, we will see that the product of a countably compact space with a Lindel\"{o}f space need not be star Menger either.

\begin{proposition}\label{products with L and CC}
     There exist a countably compact space and a Lindel\"{o}f space whose product is not star Menger.
\end{proposition}
\begin{proof}
    Let $X=[0,\mathfrak{c})$ with the usual order topology and $Y=D(\mathfrak{c})\cup\{\infty\}$ be the one-point Lindel\"{o}fication of the discrete space $D(\mathfrak{c})$ of size $\mathfrak{c}$. Let us show that the product $X\times Y$ is not star Menger. Put $D(\mathfrak{c})=\{d_\alpha:\alpha\in\mathfrak{c}\}$. For each $\alpha\in\mathfrak{c}$ let $U_\alpha=[0,\alpha]\times(\{d_\beta:\beta\geq\alpha\}\cup\{\infty\})$ and $V_\alpha=(\alpha,\mathfrak{c})\times\{d_\alpha\}$. Then, the collection $\mathcal{U}=\{U_\alpha:\alpha\in\mathfrak{c}\}\cup\{V_\alpha:\alpha\in\mathfrak{c}\}$ is an open cover of $X\times Y$. Let $M\subseteq X\times Y$ be any Menger subspace. Since $\pi_X[M]$ is Menger, then there is $\gamma\in\mathfrak{c}$ such that $M\cap[(\gamma,\mathfrak{c})\times Y]=\emptyset$. Hence, if $\alpha>\gamma$, $V_\alpha\cap M=\emptyset$ and thus, $\langle\alpha+1, d_\alpha\rangle\notin St(M,\mathcal{U})$ as $V_\alpha$ is the only element of $\mathcal{U}$ containing to $\langle\alpha+1, d_\alpha\rangle$. It follows that $X\times Y$ is not star Menger.
\end{proof}

Besides last two propositions, in \cite{Hiremath}, Hiremath gave an example of two Lindel\"{o}f spaces whose product is not star Lindel\"{o}f. Therefore, the product of Lindel\"{o}f spaces need not be star Menger either. 

\begin{proposition}[\cite{Hiremath}]\label{products with Lindelofs}
     There are two Lindel\"{o}f spaces whose product is not star Menger.
\end{proposition}

On one hand, Proposition \ref{products with cc} shows that, in particular, the product of a star Menger space and a countably compact space need not be star Menger. On the other hand, Proposition \ref{products with L and CC} shows that the product of a star Menger space and a Lindel\"{o}f space need not be star Menger. In contrast to these results, we do have that the product of a star Menger space with a compact space is star Menger. In \cite{AJW}, the authors call a property $\mathcal{P}$ compactly productive if whenever $X$ has property $\mathcal{P}$ and $Y$ is compact, then $X\times Y$ has property $\mathcal{P}$. Also, they showed that if $\mathcal{P}$ is a compactly productive property, then the star-$\mathcal{P}$ property is also compactly productive. It follows that the star Menger property is compactly productive. Moreover, since the countable union of Menger spaces is Menger, it readily follows that the star Menger property is also preserved by countable unions. So, we obtain

\begin{proposition}\label{compactly productive}
    If $X$ is a star Menger space and $Y$ is $\sigma$-compact, then the product $X\times Y$ is a star Menger space. In particular, the product of a star Menger space and a compact space is star Menger.
\end{proposition}

We can even get a small generalization of last proposition.

\begin{proposition}\label{small generalization}
    If $X$ is a star Menger space and $Y$ is a Lindel\"{o}f locally ($\sigma$-)compact space, then the product $X\times Y$ is star Menger.
\end{proposition}
\begin{proof}
    Let $\mathcal{U}$ be an open cover of $X\times Y$. For each $y\in Y$, fix an open neighbourhood $V_y$ of $y$ such that $\overline{V_y}$ is ($\sigma$-)compact. By Proposition \ref{compactly productive}, $X\times\overline{V_y}$ is star Menger, for each $y\in Y$. Hence, for each $y\in Y$, there is a Menger subspace $M_y$ of $X\times\overline{V_y}$ such that $X\times\overline{V_y}\subseteq St(M_y,\mathcal{U})$. On the other hand, since $Y$ is a Lindel\"{o}f space and the collection $\{V_y:y\in Y\}$ is an open cover of $Y$, there exists a countable subcover $\mathcal{V}$ of $\{V_y:y\in Y\}$. Let us denote $\mathcal{V}=\{V_{y_n}:n\in\omega\}$ and define $M=\bigcup\{M_{y_n}: n\in\omega\}$. Then $M$ is a Menger subspace of $X\times Y$ and $X\times Y= St(M,\mathcal{U})$. Thus, $X\times Y$ is star Menger.
\end{proof}

It is worth mentioning that Proposition \ref{small generalization} can also be obtained by noting the following immediate remark.

\begin{remark}
    Let $X$ be a locally ($\sigma$-)compact space. Then $X$ is Lindel\"{o}f if and only if $X$ is $\sigma$-compact.
\end{remark}

Regarding to continuous mappings, in \cite{AJW}, it was observed that if $\mathcal{P}$ is a topological property that is preserved under continuous images, then the star-$\mathcal{P}$ property is also preserved under continuous images. Since the Menger property is preserved by continuous images (see \cite{W}), then we obtain the following,

\begin{proposition}
    The continuous image of a star Menger space is star Menger.
\end{proposition}

Now, we may wonder about inverse images of star Menger spaces under perfect, open or closed mappings. For the cases of open and closed mappings, this property is not preserved by taking inverse images as we see below.

\begin{proposition}
    The inverse image of a star Menger space under a continuous open and closed mapping need not be star Menger.
\end{proposition}
\begin{proof}
    Let $D(\omega_1)$ be a discrete space of size $\omega_1$ and $\infty$ be a point not in $D(\omega_1)$. We consider the function $f:D(\omega_1)\rightarrow \{\infty\}$ such that for each $\alpha\in\omega_1$, $f(d_\alpha)=\infty$. It is clear that $f$ is a continuous, open, closed mapping and the space $\{\infty\}$ being a star Menger space. However, the inverse image $f^{-1}[\{\infty\}]$ which is $D(\omega_1)$ is not star Menger as it is an uncountable discrete space.
\end{proof}

It turns out that for perfect mappings we also have the same.

\begin{proposition}
    The star Menger property need not be preserved under preimages of perfect mappings.
\end{proposition}
\begin{proof}
    Let $\mathcal{A}$ be an uncountable almost disjoint family on $\omega$. We consider the space $X=A(\Psi(\mathcal{A}))\setminus(\omega\times\{1\})$ and let $\pi:X\rightarrow \Psi(\mathcal{A})$ be projection on the first coordinate. Then, $\pi$ is a perfect mapping. Now, since $\Psi(\mathcal{A})$ is separable, it is star Menger. However, its preimage $\pi^{-1}[\Psi(\mathcal{A})]$, which is $X$, is not star Menger. Indeed, note that the subset $\mathcal{A}\times\{1\}$ is an uncountable open and closed set which consist of isolated points in $X$. Thus, $\mathcal{A}\times\{1\}$ is not star Menger and, by Proposition \ref{On clopen sets}, we conclude that $X$ is not star Menger either.
\end{proof}

However, we can add an extra condition to perfect mappings so that preimages do preserve the star Menger property. Before of giving such a condition, let us show the following lemma.

\begin{lemma}\label{Menger preimages}
    The perfect preimage of a Menger space is Menger.
\end{lemma}
\begin{proof}
    Suppose $f:X\rightarrow Y$ is a perfect mapping with $Y$ being a Menger space. Let $(\mathcal{U}_n)_{n\in\omega}$ be a sequence of open covers of $X$. For each $n\in\omega$, let us denote $\mathcal{U}_n=\{U_\alpha:\alpha\in I_n\}$. Thus, for each $n\in\omega$ and for each $y\in Y$, there exists a finite subset $F_y^n$ of $I_n$ such that $f^{-1}(y)\subseteq \bigcup_{\alpha\in F_y^n}U_\alpha$. We let, for each $n\in\omega$ and for each $y\in Y$, $V_y^n=Y\setminus f[X\setminus \bigcup_{\alpha\in F_y^n}U_\alpha]$. Note that for each $n\in\omega$ and for each $y\in Y$, $V_y^n$ is an open neighbourhood of $y$. Now, we define for each $n\in\omega$, $\mathcal{W}_n=\{V_y^n:y\in Y\}$. Thus, each $\mathcal{W}_n$ is an open cover of $Y$. Since $Y$ is Menger, for each $n\in\omega$ there is a finite subcollection $\mathcal{G}_n$ of $\mathcal{W}_n$ such that $\bigcup\{\mathcal{G}_n:n\in\omega\}$ is an open cover of $Y$. Denote, for each $n\in\omega$, $\mathcal{G}_n=\{V_{y_i}^n:i\leq k(n)\}$. Hence, we have the following 
    \begin{equation*}
    \begin{split}
        X=f^{-1}[Y] &= f^{-1}\left[\bigcup\{V_{y_i}^n: i\leq k(n), n\in \omega\}\right]\\  
        &=\bigcup\{f^{-1}[V_{y_i}^n]: i \leq k(n), n\in\omega\}\\
        &=\bigcup\{f^{-1}[Y\setminus f[X\setminus\bigcup_{\alpha\in F_{y_i}^n}U_\alpha]]: i \leq k(n), n\in\omega\}\\
        &=\bigcup\{X\setminus f^{-1}[f[X\setminus\bigcup_{\alpha\in F_{y_i}^n}U_\alpha]]: i \leq k(n), n\in\omega\}\\
        &\subseteq\bigcup\{\bigcup_{\alpha\in F_{y_i}^n}U_\alpha: i \leq k(n), n\in\omega\}.
    \end{split}
    \end{equation*}
    Thus, if we define, for each $n\in\omega$, $\mathcal{V}_n=\{U_\alpha:\alpha\in F_{y_i}^n, i\leq k(n)\}$, then $\mathcal{V}_n$ is a finite subcollection of $\mathcal{U}_n$, for each $n\in\omega$, that satisfies that $\bigcup\{\mathcal{V}_n:n\in\omega\}$ is an open cover of $X$. We conclude that $X$ is a star Menger space.
\end{proof}

\begin{proposition}
    If $f:X\rightarrow Y$ is an open perfect mapping and $Y$ is a star Menger space, then $X$ is star Menger.
\end{proposition}
\begin{proof}
    Let $\mathcal{U}$ be an open cover of $X$. For each $y\in Y$, let $\mathcal{U}_y$ be a finite subcollection of $\mathcal{U}$ such that $f^{-1}(y)\subseteq\bigcup\mathcal{U}_y$ with the property that, for every $U\in\mathcal{U}_y$, $f^{-1}(y)\cap U\neq\emptyset$. Since $f$ is a closed mapping, we can choose an open neighbourhood $V_y$ of $y$ such that $f^{-1}[V_y]\subseteq\bigcup\mathcal{U}_y$, for each $y\in Y$. Moreover, given that $f$ is an open mapping, we can choose each $V_y$ so that $V_y\subseteq\bigcap_{U\in\mathcal{U}_y}f[U]$. The collection $\mathcal{V}=\{V_y:y\in Y\}$ is an open cover of $Y$ with $Y$ being a star Menger space. Therefore, there exists $N\subseteq Y$ Menger such that $St(N,\mathcal{V})=Y$. By Lemma \ref{Menger preimages}, the set $M=f^{-1}[N]$ is a Menger subspace of $X$. Let us show that $St(M,\mathcal{U})=X$. Let $x\in X$. Then, there exists $y\in Y$ such that $f(x)\in V_y$ and $V_y\cap N\neq\emptyset$. Since $x\in f^{-1}[V_y]\subseteq\bigcup\mathcal{U}_y$, there exists $U_0\in\mathcal{U}_y$ such that $x\in U_0$. Moreover, considering that $V_y\subseteq f[U]$ for each $U\in\mathcal{U}_y$, then, in particular, $V_y\subseteq f[U_0]$ and therefore, $f[U_0]\cap N\neq\emptyset$. It follows that $U_0\cap f^{-1}[N]=U_0\cap M \neq\emptyset$. Thus, $x\in St(M,\mathcal{U})$. This shows that $X$ is star Menger.  
\end{proof}

Let us finish this section by mentioning some facts on the Alexandroff duplicate. It is well-known and easy to prove that for the Menger property, we have that a space $X$ is Menger if and only if its Alexandroff duplicate $A(X)$ is Menger. Therefore, we may wonder about the behaviour of the star Menger property on the Alexandroff duplicate. To answer this question, we give the following example:

\begin{example}
    There exists a Tychonoff (pseudocompact) star Menger space $X$ such that its Alexandroff duplicate $A(X)$ is not star Menger.
\end{example}
\begin{proof}
    Let $X=\Psi(\mathcal{A})$ with $\mathcal{A}$ being a mad family on $\omega$. Then, it is obviously that $X$ is a Tychonoff pseudocompact star Menger space. Furthermore, since $\mathcal{A}$ is a closed discrete subset of $X$, then $\mathcal{A}\times\{1\}$ is an uncountable clopen subset of $A(X)$ which consists of isolated points. It follows that $\mathcal{A}\times\{1\}$ cannot be star Menger. Consequently, $A(X)$ is not a star Menger space either, as the star Menger property is inherited by clopen subsets (see Proposition \ref{On clopen sets}). 
\end{proof}

Despite the previous example, we do have the converse. That is, if $A(X)$ is star Menger, then $X$ is star Menger. In fact, something stronger holds as the following result shows; we refer the reader to \cite{SX} for details.

\begin{theorem}
    Let $X$ be a topological space. The following are equivalents:
    \begin{enumerate}
        \item[(a)] $e(X)\leq\omega$;
        \item[(b)] $e(A(X))\leq\omega$;
        \item[(c)] $A(X)$ is star countable;
        \item[(d)] $A(X)$ is star $\sigma$-compact;
        \item[(e)] $A(X)$ is star Menger;
        \item[(f)] $A(X)$ is star Lindel\"{o}f;
        \item[(g)] $A(X)$ is $SL$ (see Section \ref{section on star selection principles theory}).
    \end{enumerate}
\end{theorem}

\section{Equivalences on some classes}\label{section on equivalences on some classes}

In this section, we point out some immediate consequences of some results found in the literature which are related to the star Menger property. Some of these consequences tell us in which classes the star Menger property coincides with some other properties.

Recall that a space $X$ is said to be $\sigma$-paraLindel\"{o}f if every open cover of $X$ admits a $\sigma$-locally countable open refinement. In \cite{Hiremath}, Hiremath obtained that in the class of $\sigma$-paraLindel\"{o}f spaces, the Lindel\"{o}f and star Lindel\"{o}f properties coincide. Thus, we have the following,

\begin{proposition}
    Let $X$ be a $\sigma$-paraLindel\"{o}f space. Then, the following are equivalent:
    \begin{enumerate}
        \item[(a)] $X$ is Lindel\"{of};
        \item[(b)] $X$ is star Menger;
        \item[(c)] $X$ is star Lindel\"{o}f.
    \end{enumerate}
\end{proposition}

Recall that a development for a space $X$ is a sequence of open covers $(\mathcal{U}_n)_{n\in\omega}$ such that for each $x\in X$, the family $\{St(x,\mathcal{U}_n):n\in\omega\}$ is a local base at $x$. Thus, a Moore space is a regular space with a development. In \cite{AJMTW}, the authors proved that in the class of Moore spaces, separability and the star Lindel\"{o}f property are equivalent. As a consequence, we have,

\begin{proposition}
    If $X$ is a Moore space, then the following are equivalent:
    \begin{enumerate}
        \item[(a)] $X$ is separable;
        \item[(b)] $X$ is star Menger;
        \item[(c)] $X$ is star Lindel\"{o}f.
    \end{enumerate}
\end{proposition}

On the other hand, in \cite{AJW}, the authors investigated the question as to when a feebly Lindel\"{o}f space has countable extent. Classes of spaces where this fact occurs is in the class of $GO$ spaces (see \cite{AJW}) and in the class of normal $P$-spaces (see \cite{AJW} and also \cite{AJMTW}). In other words, in the class of $GO$ spaces as well as in the class of normal $P$-spaces, the feebly Lindel\"{o}f property and having countable extent are equivalent. Thus, we obtain the following immediate consequences:

\begin{proposition}\label{GO spaces and normal P-spaces}
    If $X$ is either a $GO$-space or a normal $P$-space, then the following are equivalent:
    \begin{enumerate}
        \item[(a)] $e(X)\leq\omega$;
        \item[(b)] $X$ is star Menger;
        \item[(c)] $X$ is feebly Lindel\"{o}f.
    \end{enumerate}
\end{proposition}

Besides the two classes mentioned above, the authors (\cite{AJW}) also showed that certain feebly Lindel\"{o}f subproducts of $\omega_1^2$ have countable extent. Namely, the product of two stationary sets of $\omega_1$ is feebly Lindel\"{o}f if and only if it has countable extent (if and only if it is normal). Therefore, we also have equivalences like in Proposition \ref{GO spaces and normal P-spaces} for this kind of products. However, for arbitrary subspaces of $\omega_1^2$, it was obtained that having countable extent is equivalent to having the star Lindel\"{of} property (\cite{AJW}). Thus, we have the following,

\begin{proposition}
    Let $X$ be a subspace of $\omega_1^2$. The following are equivalent:
    \begin{enumerate}
        \item[(a)] $e(X)\leq\omega$;
        \item[(b)] $X$ is star Menger;
        \item[(c)] $X$ is star Lindel\"{o}f.
    \end{enumerate}
\end{proposition}

Regarding to countable powers of $\omega_1$, in \cite{AJMTW}, it was obtained that same equivalence holds for subspaces of $\omega_1^\omega$. Namely, if $X$ is a subspace of $\omega_1^\omega$, then $X$ is star Lindel\"{o}f if and only if $X$ has countable extent. Hence,

\begin{proposition}
    Suppose $X\subseteq\omega_1^\omega$. Then, the following are equivalent:
    \begin{enumerate}
        \item[(a)] $e(X)\leq\omega$;
        \item[(b)] $X$ is star Menger;
        \item[(c)] $X$ is star Lindel\"{o}f.
    \end{enumerate}
\end{proposition}

Now, in relation to $C_p$-theory, the following Arkhangel'skii result is well-known,

\begin{theorem}[\cite{Ark}]\label{nw=l}
    Let $X$ be a dyadic compact space and $Y\subseteq C_p(X)$. Then $nw(Y)=l(Y)$.
\end{theorem}

As a consequence of it, we can prove the following equivalences in function spaces.

\begin{proposition}
    Suppose $X$ is a dyadic compact space. The following are equivalent:
    \begin{enumerate}
        \item[(a)] $C_p(X)$ is star countable;
        \item[(b)] $C_p(X)$ is star Menger;
        \item[(c)] $C_p(X)$ is star Lindel\"{o}f.
    \end{enumerate}
\end{proposition}
\begin{proof}
    We just need to prove $(c)$ implies $(a)$. Assume $C_p(X)$ is star Lindel\"{o}f and let $\mathcal{U}$ be an open cover of $C_p(X)$. We take a Lindel\"{of} $L\subseteq C_p(X)$ such that $St(L,\mathcal{U})=C_p(X)$. By Theorem \ref{nw=l}, $L$ has a countable network. This implies that $C_p(X)$ is star separable, that is, $C_p(X)$ is star countable.
\end{proof}

\begin{remark}
    In Theorem 1.37 of \cite{AJMTW}, it was obtained that $X$ is metrizable provided $X$ is a dyadic compact space and $C_p(X)$ is star Lindel\"{o}f.
\end{remark}

We finish this section by mentioning a last interesting fact in this topic. Every Tychonoff space can be embedded as a closed $G_\delta$ in a Tychonoff star $\sigma$-compact space (see \cite{AJW}). Therefore, as an immediate consequence of it, same situation happens for the star Menger case. For convenience of the reader, we present the idea to prove this fact.

\begin{corollary}
    If $X$ is a Tychonoff space, then $X$ can be embedded as a closed $G_\delta$ subset in a Tychonoff star Menger space.
\end{corollary}
\begin{proof}
    Let $Y=(\beta X\times\omega)\cup(X\times\{\omega\})$ with the topology inherited from $\beta X\times(\omega+1)$. Since $\omega$ is Menger, $\beta X\times\omega$ is a Menger dense subset of $Y$. Thus, $Y$ is a star Menger space. In addition, $X$ is homeomorphic to $X\times\{\omega\}$ where the last set is a closed $G_\delta$ set in $Y$. 
\end{proof}

\section{Further study and some general problems}\label{section on further study and general problems}

In this section, by looking at star kernel satisfying some selection hypotheses, we give some schemes that provide several new classes of topological spaces; some general questions are posed about them. First, we recall some classical well-known selection principles. Given an infinite set $X$, let $\mathscr{A}$ and $\mathscr{B}$ be collections of families of subsets of $X$. In \cite{MS1}, Scheepers introduced the following general forms of classical selection principles:\\
\newline
$\mathbf{S_1(\mathscr{A},\mathscr{B})}$: For any sequence $\{\mathcal{A}_n:n\in\omega\}$ of elements of $\mathscr{A}$ there is a sequence $\{B_n:n\in\omega\}$ such that for each $n\in\omega$, $B_n\in \mathcal{A}_n$ and $\{B_n:n\in\omega\}$ is an element of $\mathscr{B}$.\\
\newline
$\mathbf{S_{fin}(\mathscr{A},\mathscr{B})}$: For any sequence $\{\mathcal{A}_n:n\in\omega\}$ of elements of $\mathscr{A}$ there is a sequence $\{\mathcal{B}_n:n\in\omega\}$ such that for each $n\in\omega$, $\mathcal{B}_n$ is a finite subset of $\mathcal{A}_n$ and $\bigcup\{\mathcal{B}_n:n\in\omega\}$ is an element of $\mathscr{B}$.\\
\newline
$\mathbf{U_{fin}(\mathscr{A},\mathscr{B})}$: For each sequence $\{\mathcal{A}_n:n\in\omega\}$ of elements of $\mathscr{A}$ there is a sequence $\{\mathcal{B}_n:n\in\omega\}$ such that for each $n\in\omega$, $\mathcal{B}_n$ is a finite subset of $\mathcal{A}_n$ and $\{\bigcup \mathcal{B}_n:n\in\omega\}$ is an element of $\mathscr{B}$.\\

There are several classes of open covers for a given space $X$ (see \cite{MS1}, \cite{MS2}):
\begin{enumerate}
    \item[(1)] $\mathcal{O}$ denotes the collection of all open covers of $X$.
    \item[(2)] $\Lambda$ denotes the collection of all large covers of $X$; an open cover $\mathcal{U}$ of $X$ is large if for each $x\in X$, $x$ belongs to infinitely many elements of $\mathcal{U}$.
    \item[(3)] $\Omega$ denotes the collection of all $\omega$-covers of $X$; an open cover $\mathcal{U}$ of $X$ is an $\omega$-cover if each finite subset of $X$ is contained in some element of $\mathcal{U}$ and $X$ is not an element of $\mathcal{U}$.
    \item[(4)] $\Gamma$ denotes the collection of all $\gamma$-covers of $X$; an open cover $\mathcal{U}$ of $X$ is a $\gamma$-cover if $\mathcal{U}$ is infinite and for each $x\in X$, $x$ belongs to all but finitely many elements of $\mathcal{U}$.
\end{enumerate}

It is not difficult to note that $\Gamma\subseteq\Omega\subseteq\Lambda\subseteq\mathcal{O}$. Note that $S_{fin}(\mathcal{O},\mathcal{O})$ defines the Menger property. Other well-known classical selection properties are the Rothberger and Hurewicz properties; $S_1(\mathcal{O},\mathcal{O})$ defines the Rothberger property and $U_{fin}(\mathcal{O},\Gamma)$ defines the Hurewicz property.

The above notation motivates the following classes of star-$\mathcal{P}$ spaces:
\begin{definition}\label{star-SP spaces}
  Let $\mathscr{A},\mathscr{B}\in\{\mathcal{O}, \Lambda, \Omega, \Gamma\}$. We say that a space $X$ is 
  \begin{enumerate}
      \item star-$S_1(\mathscr{A},\mathscr{B})$ if for any open cover $\mathcal{U}$ of the space $X$, there exists a set $Y\subseteq X$ such that it satisfies the selection hypothesis $S_1(\mathscr{A},\mathscr{B})$ and $St(Y,\mathcal{U})=X$.
      \item star-$S_{fin}(\mathscr{A},\mathscr{B})$ if for any open cover $\mathcal{U}$ of the space $X$, there exists a set $Y\subseteq X$ such that it satisfies the selection hypothesis $S_{fin}(\mathscr{A},\mathscr{B})$ and $St(Y,\mathcal{U})=X$.
      \item star-$U_{fin}(\mathscr{A},\mathscr{B})$ if for any open cover $\mathcal{U}$ of the space $X$, there exists a set $Y\subseteq X$ such that it satisfies the selection hypothesis $U_{fin}(\mathscr{A},\mathscr{B})$ and $St(Y,\mathcal{U})=X$.
  \end{enumerate}
\end{definition}

Previous definition is in fact, three schemes that generate several classes of spaces; each of those three cases give us potentially sixteen topological properties\footnote{We refer the reader to \cite{MS1} and \cite{MS2} for a better understanding of selection principles and related comments.}. For instance, the class star-$S_{fin}(\mathcal{O},\mathcal{O})$ defines the star Menger property studied in this work; other classes which its investigation would be interesting are the star Rothberger property, defined by the class star-$S_1(\mathcal{O},\mathcal{O})$ and, the star Hurewicz property, defined by the class star-$U_{fin}(\mathcal{O},\Gamma)$. Moreover, we have the following general problems (when applicable). Let $\mathscr{A}$ and $\mathscr{B}$ range over the set $\{\mathcal{O}, \Lambda, \Omega, \Gamma\}$

\begin{question}
    Do there exist examples for each class of spaces star-$S_{fin}(\mathscr{A},\mathscr{B})$? Same question apply for the classes star-$S_1(\mathscr{A},\mathscr{B})$ and star-$U_{fin}(\mathscr{A},\mathscr{B})$.
\end{question}

\begin{question}
    Is it true that the classes defined from the schemes of the Definition \ref{star-SP spaces} are different each other?
\end{question}

\begin{question}
    Make an study for each property obtained from Definition \ref{star-SP spaces}.
\end{question}

\section*{Acknowledgements}

\small
\baselineskip=5pt


\textsc{Departamento de Matem\'aticas, Facultad de Ciencias, UNAM, Circuito Exterior S/N, Ciudad Universitaria, CP 04510, Ciudad de M\'exico, M\'exico.}\par\nopagebreak

\vspace{.1cm}
\textit{Email address}: J. Casas-de la Rosa: \texttt{olimpico.25@hotmail.com}

\vspace{.5cm}

\textsc{Departamento de Matem\'aticas, Facultad de Ciencias, UNAM, Circuito Exterior S/N, Ciudad Universitaria, CP 04510, Ciudad de M\'exico, M\'exico.}\par\nopagebreak

\vspace{.1cm}
\textit{Email address}: \'A. Tamariz-Mascar\'ua: \texttt{atamariz@unam.mx}

\end{document}